\documentclass[11pt,reqno]{amsart}

\usepackage{amsfonts,amsthm,amsmath,amssymb,amscd}
\usepackage{mathtools}
\usepackage{esint}
\usepackage{graphics}
\usepackage{indentfirst}
\usepackage{cite}
\usepackage{latexsym}
\usepackage[dvips]{epsfig}
\usepackage{mathrsfs}
\usepackage{hyperref}
\usepackage{color}

\allowdisplaybreaks  

\numberwithin{equation}{section}

\usepackage{graphics}
\usepackage{epsfig}

\usepackage[margin=1.1in]{geometry}

\usepackage{hyperref}
\hypersetup{hidelinks}

\newtheorem{theorem}{Theorem}[section]

\newtheorem{lemma}[theorem]{Lemma}
\newtheorem{proposition}[theorem]{Proposition}

\newtheorem{rmk}{Remark}
\numberwithin{rmk}{section}

\newenvironment{pf}{{\noindent \it \bf Proof:}}{{\hfill$\Box$}\\}
\providecommand{\abs}[1]{\left\vert#1\right\vert}
\providecommand{\norm}[1]{\left\Vert#1\right\Vert}

\def\D{\partial}
\def\DD{\nabla}
\def\Div{\text{div}}

\def\hal{\frac{1}{2}}

\def\vep{\varepsilon}
\def\ls{\lesssim}

\begin{document}

\title[Inviscid limits of Compressible Viscoelastic Equations]
{Vanishing viscosity limits of   compressible viscoelastic equations in half space}

\author{Xumin Gu}
\address
{School of Mathematics, Shanghai University of Finance and Economics, Shanghai 200433, China}
\email{gu.xumin@shufe.edu.cn.} 

\author{Dehua Wang}
\address
{Department of Mathematics, University of Pittsburgh, Pittsburgh, PA 15260, USA}
\email{dwang@math.pitt.edu.} 

\author{Feng Xie}
\address
{School of Mathematical Sciences, and CMA-Shanghai, Shanghai Jiao Tong University, Shanghai 200240, China}
\email{tzxief@sjtu.edu.cn.} 



\begin{abstract}
In this paper we consider the vanishing viscosity limit of solutions to the initial boundary value problem for compressible viscoelastic equations in the half space. When the initial deformation gradient does not degenerate and there is   no vacuum initially, we establish the uniform regularity estimates of solutions to the initial-boundary value problem for the three-dimensional compressible viscoelastic equations in the Sobolev spaces. Then we justify the vanishing viscosity limit of solutions of the  compressible viscoelastic equations based on the uniform regularity estimates and the compactness arguments. Both the no-slip boundary condition  and the Navier-slip type boundary condition on velocity are addressed in this paper. On the one hand, for the corresponding vanishing viscosity limit of the compressible Navier-Stokes equations with the no-slip boundary condition, it is impossible to derive such uniform energy estimates of solutions due to the appearance of strong boundary layers. Consequently, our results show that the deformation gradient can prevent the formation of strong boundary layers. On the other hand, these results also provide two different kinds of suitable boundary conditions for the well-posedness of the initial-boundary value problem of the elastodynamic equations via the vanishing viscosity limit method. Finally, it is worth noting that we take advantage of the Lagrangian coordinates   to study the vanishing viscosity limit for the fixed boundary problem in this paper.

\end{abstract}

\date{\today}

\maketitle

\section{Introduction}
\subsection{Formulation in the Eulerian coordinates}
In this paper we are concerned with the vanishing viscosity limit of solutions to the compressible isentropic neo-Hookean viscoelastic fluids in the half space
$\mathbb{R}^3_+=\{(x, y, z)\in \mathbb{R}^3: (x, y)\in \mathbb{R}^2, z\geq 0\}$, governed by the following equations:
\begin{equation}\label{FVEE}
	\begin{cases}
		\D_t \rho^\varepsilon+\Div(\rho^\varepsilon u^\varepsilon)=0,\\
		\D_t(\rho^\varepsilon u^\varepsilon)+\Div(\rho^\varepsilon u^\varepsilon\otimes u^\varepsilon)+\DD p(\rho^\varepsilon)-\Div T^\varepsilon=\Div(\rho^\vep F^\vep {F^\vep}^\top),\\
		\D_t F^\vep+u^\vep\cdot \DD F^\vep=\DD u^\vep F^\vep,
	\end{cases}
\end{equation}
where $\rho^\vep$ denotes the density, $u^\varepsilon=(u_1^\varepsilon,u_2^\vep,u_3^\varepsilon)^\top\in\mathbb{R}^3$ is the velocity, and $F^\varepsilon=(F^\varepsilon_{ij})\in\mathbb{M}^{3\times 3}$ is the deformation gradient.
The pressure $p(\rho^\varepsilon)$ takes the following form: 
\begin{equation}\label{gamma-law}
	p(\rho^\vep)=A(\rho^\varepsilon)^\gamma,\quad\gamma>1,
\end{equation}
and without loss of generality we assume $A=1$ 
for  simplicity.  The viscous stress tensor $T^\vep$ is of the form
\begin{equation}\label{vis-tensor}
	T^\vep=2\mu\varepsilon Su^\varepsilon+\lambda\varepsilon \Div u^\varepsilon{I},
\end{equation}
where $Su^\varepsilon=\frac{1}{2}(\DD u^\varepsilon+(\DD u^\vep)^\top)$ is the symmetric part of $\DD u^\varepsilon$, and $\mu\vep$, $\lambda\vep$ are two viscosity coefficients which satisfy the physical constrains that $\mu>0$ and $2\mu+3\lambda>0$.
In the paper $\top$ denotes the transpose of a matrix.

We will study the vanishing viscosity limit of solutions to the initial-boundary value problem for the system \eqref{FVEE}-\eqref{vis-tensor} with the initial data:
\begin{equation}\label{V-Ini-D}
	(\rho^\vep, u^\vep, F^\vep)|_{t=0}=(\rho_0^\vep,u_0^\vep,F_0^\vep),
\end{equation}
%
under the following two types of boundary conditions:

\noindent {\bf Type I.} No-slip boundary condition:
 \begin{equation}\label{VKBC}
	u^\varepsilon=0,\text{ on } \{z=0\}.
\end{equation}
{\bf Type II.} Navier-slip boundary condition:
\begin{equation}
\label{VKBCN}
        u^\vep\cdot n=0,\,\, (\rho^\vep F^\vep F^{\vep \top}n+2\mu\vep{Su}^\vep n)_\tau=-\alpha\varepsilon u^\varepsilon_\tau,\,\,\text{on } \{z=0\},
\end{equation}
where $f_\tau=f-(f\cdot n) n$ and $n=(0, 0, -1)$ is the unit outward normal of the boundary of $\mathbb{R}^3_+$.

Formally, when $\vep\rightarrow 0$, the system \eqref{FVEE}--\eqref{V-Ini-D} can be reduced to the compressible elastodynamic  system:
\begin{equation}\label{FEE}
\begin{cases}
\D_t \rho+\Div(\rho u)=0,\\
\D_t(\rho u)+\Div(\rho u\otimes u)+\DD p(\rho)=\Div(\rho FF^\top),\\
\D_t F+ u\cdot \DD F=\DD u F,\\
 (\rho,u,F)|_{t=0}=(\rho_0,u_0,F_0).
\end{cases}
\end{equation}
The aim of this paper is to justify rigorously that the solutions of the initial-boundary value problem \eqref{FVEE}--\eqref{V-Ini-D} for the viscoelastic flows under the two different   boundary conditions \eqref{VKBC} and \eqref{VKBCN} converge to the solutions of the elastodynamic equations \eqref{FEE} of inviscid   flows as the viscosity tends to zero. As a direct application, we will give two kinds of suitable boundary conditions for the well-posedness of the initial-boundary value problem for the elastodynamic equations via the vanishing viscosity  method.

\subsection{Reformulation in the Lagrangian coordinates}
We shall  derive   {\it a priori} estimates that are uniform in $\vep$ for the solutions to the initial-boundary value problem \eqref{FVEE}--\eqref{V-Ini-D} under two different  boundary conditions   \eqref{VKBC} and   \eqref{VKBCN}, respectively. Instead of the Eulerian coordinates, we shall study the problem under the Lagrangian coordinates, as widely used in elastodynamics (system \eqref{FVEE} for $\vep=0$).

Let $\eta^\vep(x,t)\in \mathbb{R}^3_+$ be the ``position" of the compressible viscoelastic fluid particle $x$ at time $t$, that is,
\begin{equation}\label{flow map}
\begin{cases}
	\partial_t\eta^\vep(x,t)=u^\vep(\eta^\vep(x,t),t),& \text{ for }t>0, \ x\in \mathbb{R}^3_+,\\
	\eta^\vep(x,0)=\eta_0^\vep(x), &\text{ for }x\in \mathbb{R}^3_+,
\end{cases}
\end{equation}
where $\eta^\vep_0$ is a diffeomorphism mapping from the reference domain $\mathbb{R}^3_+$ to the initial domain $\mathbb{R}^3_+$ satisfying
\begin{equation}\label{Ini-flow map}
F^\vep_0(\eta^\vep_0)=\nabla\eta^\varepsilon_0.
\end{equation}
We introduce the Lagrangian variables as follows:
\begin{align*}
f^\vep(x,t)=\rho^\varepsilon(\eta^\vep(x,t),t),~~v^\vep(x,t)=u^\vep(\eta^\vep(x,t),t), ~~G^\vep(x,t)=F^\varepsilon(\eta^\vep(x,t),t),
\end{align*}
and set
\begin{align*}
A^\vep=(\DD\eta^\vep)^{-\top},~~J^\vep=\text{det} \DD\eta^\vep,~~ a^\vep=J^\vep A^\vep,~~q(f^\vep)=p(f^\vep).
\end{align*}
 Then, after using the chain rule and Einstein's summation convention for the repeated indices, the compressible viscoelastic fluid equations \eqref{FVEE}--\eqref{V-Ini-D} can be formulated in terms of the Lagrangian variables in $\mathbb{R}^3_+$ as follows:
\begin{equation}\label{FEL}
\begin{cases}
\D_t\eta^\vep=v^\varepsilon,\\
\D_tf^\vep+f^\vep \Div_\eta v^\vep=0,\\
f^\vep\D_tv^\vep_i+\partial_{\eta_i}q(f^\vep)-2\mu\vep \partial_{\eta_k}(S_\eta (v^\vep))_{ik}-\lambda\vep \partial_{\eta_i}(\Div_\eta v^\vep)-\partial_{\eta_k}(f^\vep G^\vep_{ij}G^\vep_{kj})=0,\\
\D_tG^\vep_{ij}=\D_{\eta_k}v^\vep_iG^\vep_{kj},\\
(f^\vep,v^\vep,G^\vep,\eta^\vep)(x)|_{t=0}=(\rho^\vep_0(\eta^\vep_0(x)), u^\vep_0(\eta^\vep_0(x)), F^\vep_0(\eta^\vep_0(x)),\eta^\vep_0(x)),
\end{cases}
\end{equation}
where
$$\D_{\eta_i}=A^\vep_{i\ell}\D_\ell, \quad (S_\eta v)^\vep_{ik}=\frac{1}{2}(\D_{\eta_k}v^\vep_i+\D_{\eta_i}v^\vep_k), \quad\Div_\eta v^\vep=\D_{\eta_k}v^\vep_k.$$
From Jacobi's formula
\begin{equation}\label{Jacobi's formula}
\D_t J^\vep=J^\vep A^\vep_{kj}\D_jv^\vep_k=J^\vep\Div_\eta v^\vep,
\end{equation}
the second equation in \eqref{FEL} can be directly solved to obtain $\D_t(f^\vep J^\vep)=0$, which implies
\begin{equation}\label{f-eqn}
f^\vep=\tilde{\rho}^\vep_0(J^\vep)^{-1},\text{ and }\tilde{\rho}^\vep_0:=\rho^\vep_0(\eta^\vep_0)J^\vep_0.
\end{equation}
Next, from the fact that $\D_t A^\vep_{ij}=-A^\vep_{ik}\D_kv^\vep_lA^\vep_{lj}$ and the fourth equation in \eqref{FEL}, it is straightforward  to verify that $\D_t({A^\vep}^\top{G}^\vep)=0,$
which yields
\begin{equation}\label{G-eqn}
{G}^\vep=\DD \eta^\vep(\DD\eta^\vep_0)^{-1}F^\vep_0(\eta_0^\vep)=\DD\eta^\vep.
\end{equation}
Moreover,  noticing that ${a}^\vep$ is the cofactor of  $\nabla\eta^\vep$   one has the following Piola identity:
\begin{equation}\label{Piola}
\D_l a^\vep_{kl}=\D_l(J^\vep A^\vep_{kl})=0,
\end{equation}
which leads to
\begin{equation*}
-A^\vep_{kl}\D_l(f^\vep G^\vep_{ij}G^\vep_{kj})=-(J^\vep)^{-1}\D_l(A^\vep_{kl}f^\vep J^\vep G^\vep_{ij}G^\vep_{kj}).
\end{equation*}
Therefore, based on 
\eqref{f-eqn} and \eqref {G-eqn}, it suffices to consider the following system of equations rather than the system \eqref{FEL},
\begin{equation}\label{FEL-final}
\begin{cases}
\D_t\eta^\vep_i=v^\vep_i,\\
\tilde{\rho}_0^\vep\D_tv^\vep_i+a_{ik}^\vep\D_kq^\vep-2\mu\vep a^\vep_{kl}\D_l(S_\eta v)^\vep_{ik}-\lambda\vep a^\vep_{ij}\D_j(\Div_\eta v^\vep)-\D_j(\tilde{\rho}_0^\vep\D_j\eta^\vep_i)=0,\\
(v^\vep,\eta^\vep)(x)|_{t=0}=(u_0^\vep(\eta^\vep_0(x)),\eta^\vep_0(x)).
\end{cases}
\end{equation}
Taking $\vep=0$ in \eqref{FEL-final} formally  leads to the following elastodynamic equations in the Lagrangian coordinates:
\begin{equation}\label{FEL-final-2}
	\begin{cases}
		\D_t\eta_i=v_i,\\
		\tilde{\rho}_0\D_tv_i+a_{ik}\D_kq-\D_j(\tilde{\rho}_0\D_j\eta_i)=0,\\
		(v,\eta)(x)|_{t=0}=(u_0(\eta_0(x)),\eta_0(x)).
	\end{cases}
\end{equation}

We now remark   why we introduce the Lagrangian flow map for the fixed boundary problem.
Notice that the deformation gradient effect is in fact related to the motion of the fluid itself. Precisely, if we use the Lagrangian flow map, the system of equations \eqref{FVEE} can be rewritten in a relatively concise form of \eqref{FEL-final} with the boundary   unchanged under the both  boundary conditions \eqref{VKBC} and \eqref{VKBCN}.

Under the Lagrangian coordinates, the corresponding boundary conditions can be written as the following:

\noindent
{\bf Type I.} No-slip boundary condition:
\begin{equation}
        \label{bdDirichlet}
        v_i^\varepsilon=0\quad (i=1,2,3),\text{ on } \{z=0\}.
\end{equation}
{\bf Type II.} Navier-slip boundary condition:
\begin{equation}
        \label{bdNavier}
		\begin{cases}
			&v_3^\varepsilon=0,\,\, a^\varepsilon_{\cdot 3}=(0,0,a^\varepsilon_{33})\text{ on } \{z=0\},\\&\tilde\rho_0^\vep F^\vep_{i3}\abs{a_{\cdot 3}^\vep}^2-\tilde{\rho}_0^\vep J^\vep a_{i3}^\vep+2\mu\vep\left((S_\eta(v^\vep))_{i3}a_{33}^\vep\abs{a_{\cdot 3}^\vep}^2-(S_\eta(v^\vep))_{33}(a_{33}^\vep)^2a_{i3}^\vep\right) \\
			&\qquad\qquad\qquad =-\alpha \vep v_i^\vep\abs{a_{\cdot 3}^\vep}^3,\text{ on } \{z=0\}.
		\end{cases}
\end{equation}
The boundary condition \eqref{bdNavier} indicates that
\begin{equation}
	\label{340}
        -\alpha \vep \abs{a_{33}^\vep} v_\beta^\vep=\tilde\rho_0F^\vep_{\beta 3}+2\mu\vep(S_\eta(v^\vep))_{\beta 3}a_{33}^\vep,\,\,\text{ on } \{z=0\}, \,\beta=1,2.
\end{equation}
Before introducing the main results in this paper, let us review the related known results briefly. First, the viscoelastic system of equations is a fundamental system in complex fluids. The well-posedness of solutions to both the compressible and incompressible viscoelastic equations have been extensively studied, see \cite{LLZ2,Hu-Lin-Liu2018, L,LZ,HZ, HW2,Hu-Lin2016} and the references therein. Moreover, the vanishing viscosity limit is also one of the important problems in hydrodynamics and applied mathematics,  and has  attracted much attention from mathematicians. The vanishing viscosity limit of solutions to the Cauchy problem has been studied in many works; see \cite{CW,K,Mas,S} for the incompressible Navier-Stokes equations,  \cite{Cai,Sideris22} for the incompressible viscoelastic equations, and \cite{HL89,HWY-2} and their references   for the compressible Navier-Stokes equations.

When the vanishing viscosity limit problem is considered in a domain with a physical boundary,
it usually becomes more challenging due to the presence of boundary layers (e.g.  \cite{GGW,Ole,sch,von,WXY}). In particular,
if there exists a strong boundary layer,  the vanishing viscosity limit problem   is extremely hard due to the uncontrollability of the vorticity of boundary layer correctors.
However, when the no-slip boundary condition is replaced by the Navier-slip type boundary condition, the strong boundary layer usually disappears for the 
hydrodynamics equations, and the vanishing viscosity limit has been proved in \cite{WW, WXY1} for the compressible Navier-Stokes equations.  We refer the readers to \cite{B, IS, MR, XX,CLQ18} and their references   for the corresponding vanishing viscosity limit of the incompressible Navier-Stokes equations with the Navier-slip boundary conditions.
%
%
However, the research of the vanishing viscosity limit problem in a domain with a boundary under the  no-slip boundary condition is relatively underdeveloped.
To our best knowledge, the vanishing viscosity limit of the unsteady incompressible Navier-Stokes equations with the no-slip boundary condition  was only proved in the analytic function spaces or in the Gevrey classes; see  \cite{samm-caf1,samm-caf2, Mae, GMM} and the references therein for more details.
  For the incompressible magnetohydrodynamic (MHD) equations with the no-slip boundary condition on velocity, the existence and uniqueness of solutions to the MHD boundary layer equations and the convergence of the Prandtl boundary layer expansion in the Sobolev framework  were established in \cite{L-X-Y1, L-X-Y2} when the tangential component of magnetic field has a positive lower bound near the physical boundary initially; and it was proved in \cite{L-X-Y3} that the strong boundary layers do not happen in the vanishing viscosity limit for the incompressible non-resistive MHD system when
  the magnetic field is transversal to the physical boundary initially. This phenomenon was also observed for the compressible non-resistive MHD system and the MHD system with magnetic diffusion in \cite{CLX1,CLX2}.
  However, the vanishing viscosity limit for the compressible Navier-Stokes equations under the no-slip boundary condition in the half plane in the Sobolev spaces is still open, except for the linearized Navier-Stokes equations \cite{XY1} and  the case considered in the analytic function spaces \cite{WWZ}; and  even in the Gevrey settings the vanishing viscosity limit  is also unclear because of the appearance of strong boundary layers \cite{GGW,WXY}.

  In this paper, we consider the vanishing viscosity limit of the compressible viscoelastic equations in the half space under two different kinds of boundary conditions.
We show that the deformation tensor in viscoelasticity produces a significant effect on the vanishing viscosity process, and it can prevent the generation of strong boundary layers even under the no-slip boundary conditions on velocity. For this reason we are able to justify the vanishing viscoity limit of solutions to the compressible viscous flows governed by the viscoelastic equations \eqref{FVEE} in a domain with a physical boundary. Moreover, the main results in this paper also give two kinds of suitable boundary conditions for the well-posedness of initial boundary value problem of elastodynamic equations by the method of vanishing viscosity limit.
The similar effect of the deformation tensor was also observed in \cite{gu2020,gu2022} for the two-dimensional free boundary problem. The vanishing viscosity limit of solutions to the compressible viscoelastic equations with the no-slip boundary condition on velocity in the two-dimensional half plane was  considered in \cite{WX}, where all of desired energy estimates are derived in the Eulerian coordinates, and the working spaces are co-normal Sobolev spaces. Here, we use  the Lagrangian flow map method to derive all necessary energy estimates in the classical Sobolev spaces. Moreover, this vanishing viscosity limit problem in  both the two-dimensional and the  three-dimensional cases can be solved at the same time for the two different kinds of boundary conditions including the no-slip boundary condition  and Navier-slip boundary condition  in this paper.

\subsection{Main results}

Before stating the main results, we introduce the notations that will be used frequently throughout the paper.
 First  Einstein's summation convention will be adopted for repeated indices.   We denote by $\D_\tau^\alpha:=\D_x^{\alpha_1}D_y^{\alpha_2}\ (\alpha=(\alpha_1, \alpha_2),\ |\alpha|=\alpha_1+\alpha_2)$   the tangential derivatives, $\D_t$ denotes the temporal derivative and $\D$ also includes the normal derivative $\D_z$. The standard $L^p$ spaces, Sobolev spaces $H^m=W^{m,2}$ and $W^{m,p}$ on both the domain $\mathbb{R}^3_+$ and its boundary $\Gamma:=\{z=0\}$ are used. For the sake of simplicity, the norms for these function spaces defined on $\mathbb{R}^3_+$ are written as $\|{\cdot}\|_{L^p}, \|{\cdot}\|_{W^{m,p}}$ and $\|{\cdot}\|_{m}$, and the norms for these function spaces defined on $\Gamma$ are denoted by $|{\cdot}|_{L^p}, |{\cdot}|_{W^{m,p}}$ and $|{\cdot}|_{m}$.  For any real $s\geq 0$, the Hilbert space $H^s(\Gamma)$ and the related boundary norm $|{\cdot}|_s$ (or $|{\cdot}|_{H^s(\Gamma)}$) are defined by interpolation. The negative-order Sobolev space $H^{-s}(\Gamma)$ can be understood as duality: for real $s\geq 0, H^{-s}(\Gamma):=[H^{s}(\Gamma)]^\prime.$
The   norm of the space $L^p([0, t];\textbf{X})$ is denoted by $\|\cdot\|_{L^p_t(\textbf{X})}$.  We shall denote by $C$   a generic constant
that  depends only  on the domain $\Omega$ and the boundary $\Gamma$.   The notation $f\lesssim g$ means $f\leq Cg$.  We shall use $P(\cdot)$ to denote  a generic polynomial function of its arguments, which may vary from line to line, but the polynomial coefficients are generic constants $C$  independent of $\varepsilon$.

The aim of this paper is to establish the   well-posedness  and the vanishing viscosity limit of classical solutions to \eqref{FEL-final}  in a fixed time interval  independent of the viscosity $\vep\in(0,1]$. To obtain the uniform regularity estimates, we define the energy functional for the solutions to the viscoelastic fluid equations \eqref{FEL-final} as
\begin{align}\label{enE}
        \mathfrak{E}^\vep(t)=&\sum_{j=0}^m\norm{\partial_t^j\eta^\vep}_{m-j}^2(t)+\|\partial_t^j\partial_\tau^{m-j}(\nabla\eta^\vep, v^\vep, q^\vep)\|_{0}^2(t)+\sum_{j=0}^{m-1}\vep\|\partial_t^j\DD^2\eta^\vep\|_{m-1-j}^2(t)\nonumber\\
	&+\sum_{j=0}^m\int_{0}^{t}\norm{\D_t^j(\DD\eta^\vep, v^\vep,\vep\nabla v^\vep)}_{m-j}^2+\norm{\sqrt\vep\D_t^j\D_\tau^{m-j}\DD v^\vep}_{0}^2\,d\mathsf t.
\end{align}
We also require that the initial data $\D_t^\ell v^\vep(0),\D_t^\ell\eta^\vep (0), \D_t^\ell q^\vep(0)$, $\ell=0,\cdots,m-1$ satisfy the compatibility condition on the boundary. That is, for the no-slip boundary condition, the compatibility condition is
 \begin{equation}\label{compat-cond-d}
 	\D_t^\ell v^\vep(0)=0, \text{ on } \Gamma,
 \end{equation}
where $\D_t^\ell v^\vep(0),\D_t^\ell\eta^\vep(0), \D_t^\ell q^\vep(0)$ are defined by
\begin{align*}
	&\D_t^\ell v_i^\vep(0)=\tilde{\rho}_0^{-1}\D_t^{\ell-1}(-a_{ik}^\vep\D_kq^\vep+2\mu\vep a^\vep_{kl}\D_lS_\eta(v)^\vep_{ik}+\lambda\vep a^\vep_{ij}\D_j\Div_\eta (v)^\vep+\D_j(\tilde{\rho}_0^\vep\D_j\eta^\vep_i))\big|_{t=0},\\
	&\D_t^\ell\eta^\vep_i(0)=\D_t^{\ell-1}v_i^\vep(0),\quad \D_t^\ell q^\vep(0)=\tilde{\rho}_0^\gamma\D_t^\ell(J^\vep)^{-\gamma}\big|_{t=0}.
\end{align*}
While for the Navier-slip boundary condition, the compatibility condition becomes
\begin{equation}\label{compat-cond-n}
	\D_t^\ell v_3^\vep(0)=0,\quad  \D_t^\ell\left(\alpha \vep \abs{a_{33}^\vep} v_\beta^\vep+\tilde\rho_0\D_3\eta^\vep_\beta+2\mu\vep(S_\eta(v^\vep))_{\beta 3}a^\vep_{33}\right)|_{t=0}=0, \beta=1,2 \text{ on } \Gamma.
\end{equation}
The local well-posedness and uniform regularity estimates of solutions to the initial-boundary value problem \eqref{FEL-final}, \eqref{bdDirichlet} or \eqref{bdNavier} can be stated in the following theorem.

\begin{theorem}\label{Theorem 1}  Let $m\geq 4$. Suppose that the initial data $(\rho_0^\vep, \eta_0^\vep, v_0^\vep)$ satisfies the compatibility condition \eqref{compat-cond-d} (or the compatibility condition \eqref{compat-cond-n})  and   the following uniform bounds:
	\begin{align}
        &0< c_0\leq \tilde{\rho}_0^\vep\leq C_0,\label{rho-uni-bdd}\\ &\mathfrak E^\vep(0)\leq C_0
	\end{align}
for some generic constants $c_0$ and $C_0$. Then, there exists a $T_0>0$ independent of $\vep$, and a unique solution $(\eta^\vep,v^\vep)$ to the initial-boundary value problem \eqref{FEL-final} with \eqref{bdDirichlet} (or with   \eqref{bdNavier}) on the time interval $[0,T_0]$, such that
	\begin{equation}
		\sup_{t\in[0,T_0]}\mathfrak{E}^\vep(t)\leq C_1,
	\end{equation}
	where $C_1$ is a generic constant depending only on $c_0,C_0$.
\end{theorem}

\begin{rmk}
	The regularity of solutions implies that the flow map $\eta$ is at least Lipschitz continuous, thus one can recover the corresponding classical solutions to \eqref{FVEE} in the Eulerian coordinates.
\end{rmk}
Based on the uniform regularity estimates achieved in Theorem \ref{Theorem 1}, 
we can justify the vanishing viscosity limit of solutions to the initial-boundary value problem for the compressible viscoelastic flow and obtain the local existence of classical solutions to the related initial-boundary value problem for the elastodynamic equations with two different kinds of boundary conditions.

\begin{theorem}\label{vanishing viscosity1}
	Under the assumptions of Theorem \ref{Theorem 1}, if we further assume  that there exists $(\rho_0,\eta_0,v_0)$ such that
	\begin{equation}
		\lim\limits_{\vep\rightarrow 0}\|\tilde{\rho}^\vep_0-\rho_0\|_0+\|\eta_0^\vep-\eta_0\|_0+\|v_0^\vep-v_0\|_0=0.
	\end{equation}
Then, there exists $(\eta,v)(t,\cdot)$ on the time interval $[0,T_0]$ such that
\begin{equation}
	\sup_{t\in[0,T_0]}\mathfrak{E}(t)\leq C_1,
\end{equation}
and
\begin{equation}
	\lim\limits_{\vep\rightarrow 0}\sup_{t\in[0,T_0]}(\|\eta^\vep(t)-\eta(t)\|_{m}+\|v^\vep(t)-v(t)\|_{m-1})=0.
\end{equation}
Moreover, $(\eta,v)$ is the unique classical solution to the initial-boundary value problem of the elastodynamic equations \eqref{FEL-final-2} with the following boundary condition:
\begin{align}
v=0 \qquad (\text{or} \quad  v_i a_{i3}=0,\qquad \tilde\rho_0F_{i3}\abs{a_{\cdot 3}}^2-\tilde{\rho_0}J a_{i3} =0)
\end{align}
on $\Gamma$.
\end{theorem}

\begin{rmk}
In fact, Theorem \ref{vanishing viscosity1} 
provides two different kinds of boundary conditions for the well-posedness of the initial-boundary value problem of the elastodynamic equations \eqref{FEL-final-2} by the vanishing viscosity limit method. In the Eulerian coordinates, they can be written as $u=0$ (or  $u\cdot n=0$, $(\rho F F^{\text{T}}n)_\tau=0$).
\end{rmk}

\subsection{Comments on the main results and strategies of the proofs}
We now explain the main difficulties and the strategies of the proofs of  the main theorems.
As mentioned above,  when the vanishing viscosity limit is considered in a domain that has a physical boundary, the uniform estimates of normal derivatives of solutions with respect to the small viscosity parameter $\varepsilon$ are difficult to achieve. On the one hand, in general it is impossible to obtain these uniform estimates due to the appearance of strong boundary layers 
for the solutions to both the compressible and incompressible Navier-Stokes equations under the no-slip boundary condition on velocity.
On the other hand, when the uniform estimates of normal derivatives of solutions with respect to the small viscosity parameter $\varepsilon$ are derived, it shows that the strong boundary layer should disappear. In our results we find that if the deformation gradient in viscoelasticity is taken into account, even though the no-slip boundary condition is given on the velocity, the uniform estimates of normal derivatives for the solutions to the compressible viscoelastic fluid equations can still   be achieved provided that the deformation gradient does not degenerate. In fact, the combination of deformation gradient and the pressure gives a good positive structure for the normal derivatives of the flow map $\eta$, which is one of the main observations of this paper. 
This structure helps us to control the normal derivatives uniformly with respect to the viscosity.

Briefly, by using the formula of $q$ and \eqref{FEL-final}, we have
\begin{equation}\label{e4}
	-\mathcal{A}_{ij}\D_3^2\eta_j^\vep-\mu\vep a^\vep_{k3}a^\vep_{k3}\D_3^2v_i^\vep-(\mu+\lambda)\vep a_{i3}^\vep a_{j3}^\vep \D_3^2v_j^\vep= \mathcal{F}_i+\mathcal{G}_i,
\end{equation}
where $\mathcal{F}_i, \mathcal{G}_i$ only contain at most one normal derivatives of $\eta^\vep, \vep v^\vep$,  and
\begin{equation}\label{e5}
	\mathcal A={\tilde\rho}_0^\vep J^\vep(I+\gamma(\tilde\rho_0^\vep)^{\gamma-1}(J^\vep)^{-\gamma-1}n^\vep\otimes n^\vep), 
\end{equation}
where $n^\vep=a_{\cdot 3}^\vep$ denotes the third column of $a^\vep$.
Thus, with $0<c_0\leq \tilde\rho_0^\vep, J^\vep \leq C_0$, it is straightforward to justify that $\mathcal{A}$ is positive definite and show that $\D_3^2\eta^\vep$ can be controlled by the lower order normal derivatives of $\eta^\vep$ uniformly with respect to $\vep$.

Consequently, according to the above arguments, the results stated in Theorem \ref{Theorem 1} show that the strong boundary layers will disappear when the non-degenerate deformation gradient is involved in the viscous flow. Moreover, our method is also valid for the initial-boundary value problem for the compressible viscoelastic  equations with the Navier-slip type condition, which is also studied in this paper, and similar results can be obtained as stated in Theorem \ref{Theorem 1} and Theorem \ref{vanishing viscosity1}.


The remainder of the paper is organized as follows. In Section 2, we present some preliminaries and elementary lemmas. Section 3 is devoted to deriving the uniform energy estimates of solutions to the initial-boundary value problem \eqref{FEL-final} and \eqref{bdNavier}.
In Section 4, we establish the uniform estimates of solutions to the initial-boundary value problem \eqref{FEL-final} and \eqref{bdDirichlet}. Based on the uniform estimates established in Sections 3 and 4, we prove the main Theorems   \ref{Theorem 1}  and \ref{vanishing viscosity1}  in Section 5.

\section{Preliminaries}

In this section, we recall some basic inequalities and   identities and   estimates.

\begin{lemma}\label{jj}
	Let $g\in H^1([0,t];L^2)$. Then, we have
	\begin{equation}\label{L-infL-2}
		\|g(t)\|_0^2\ls t\|\D_tg\|_{L^2_t(L^2)}^2+\|g(0)\|_0^2.
	\end{equation}
\end{lemma}
\begin{proof}
	Since $g\in H^1([0,t];L^2)$, it follows that $g(t,x)\in C([0,t];L^2)$. The elementary theorem in calculus implies that
	\begin{equation*}
		g(t,x)=g(0, x)+\int_{0}^{t}g_t(s,x)ds.
	\end{equation*}
Thus, after using the H\"{o}lder and Minkowski inequalities, we arrive at \eqref{L-infL-2}.
\end{proof}

%
%

\subsection{Trace estimates}
The following trace estimates will be used.
\begin{lemma} For any function $\theta\in H^1$, one has
\begin{equation}\label{tre}
\abs{\theta}_0^2\ls \norm{\theta}_0^2+\norm{\theta}_0\norm{\nabla \theta}_0.
\end{equation}
\end{lemma}
\begin{lemma}\label{pereaf}
	Denote the dual space of $H^{\frac{1}{2}}(\Gamma)$ by $H^{\frac{1}{2}}(\Gamma)^\prime$. Then the following inequality holds true
	\begin{equation}
	\label{peadv}
	\abs{\partial_\tau \omega}_{-\frac{1}{2}}:=\abs{\partial_\tau \omega}_{H^{\frac{1}{2}}(\Gamma)^\prime}\leq C\abs{\omega}_{\frac{1}{2}}, \; \forall \omega\in H^{\frac{1}{2}}(\Gamma).
	\end{equation}
\end{lemma}
\begin{proof}
	The proof can be found in \cite[Lemma 8.5]{DS_10}.
\end{proof}
\subsection{Korn's inequality}
We refer to \cite{MR} for the following Korn-type inequality.
\begin{lemma}
	For any $f\in H^1(\Omega)$, it holds that
	\begin{equation}\label{Korn's ineq}
            \|\DD f\|_0^2\ls P(\|\DD\eta\|_{2}^2)(\|{S}_{\eta}(f)-\dfrac{1}{3}\Div_\eta(f)I\|_0^2+\|f\|_0^2).
	\end{equation}
\end{lemma}

\subsection{Geometric identities}
By the definitions and the chain rule, differentiating $J^\vep$, $A^\vep$ and $a^\vep$, we obtain
\begin{equation}\label{Geo-iden-1}
\D J^\vep=a_{ij}^\vep\D_j\D\eta_i^\vep,\quad \D A^\vep_{kj}=-A_{kl}^\vep\D_l \D\eta_i^\vep A_{ij}^\vep,\quad \D a^\vep_{kj}=a_{li}^\vep\D_i\D\eta_lA^\vep_{kj}-a^\vep_{kl}\D_l\D\eta^\vep_iA^\vep_{ij}.
\end{equation}


\section{Viscosity-Independent a priori Estimates}

In this section, we focus on deriving the $\vep$-independent estimates of smooth solutions to \eqref{FEL-final} under two different types boundary conditions, which are stated in the following proposition.
\begin{proposition}\label{uniform estimates}
	Let $(\eta^\vep,v^\vep)$ be a solution to \eqref{FEL-final} with the no-slip boundary condition \eqref{bdDirichlet} or the Navier-slip boundary condition \eqref{bdNavier}. Then there exists a time $T$ independent of $\vep$ such that
\begin{equation}
\sup_{t\in[0,T]}\mathfrak{E}^\vep(t)\leq 2M_0,
\end{equation}
where $M_0=P(\mathfrak{E}^\vep(0))$.
\end{proposition}

Note that,  since
\begin{equation}\label{H1}
\tilde{\rho}_0^\vep\geq c_0,~~ \frac{1}{c_0}>J^\vep_0\geq c_0>0
\end{equation}
for some $c_0>0$, we can assume that there exist a sufficiently small $T_\vep$ such that, for $t\in[0,T_\varepsilon]$,
\begin{equation}\label{A-priori-assum}
|J^\vep(t)-J^\vep_0|\leq \frac{1}{8}c_0,\quad|\D_j\eta_i^\vep(t)-\D_j\eta_{0i}^\vep|\leq\frac{1}{8}c_0.
\end{equation}
Note that the lower order terms in $\mathfrak E^\vep$ can be estimated directly. Indeed, using Lemma \ref{jj}, we have
for $t\in [0,T]$ with $T\le T_{\vep}$,
\begin{equation}\label{etaest}
\mathfrak F^\vep(t):=\sum_{j=0}^m\norm{\D_t^j(\eta^\vep,\vep\nabla\eta^\vep)}_{m-j}^2(t)\ls M_0+T\mathfrak E^\vep(T).
\end{equation}

The proof of Proposition \ref{uniform estimates} can be divided into the proofs of several lemmas. We will prove the proposition   for the Navier-slip boundary condition  in details first, and then   extend the proof to the no-slip boundary condition case.
For simplicity of notation, we only keep the superscript $\vep$ 
in the statements of lemmas but omit it in the proofs without causing any confusion.



\subsection{Basic energy estimates}
\begin{lemma}\label{basic-energy-est}
	For any $t\in[0,T_\vep]$, 
	\begin{equation}\label{Basic energy}
	\|v^\vep(t)\|_0^2+\|\DD\eta^\vep(t)\|_0^2+\|Q(f^\vep)(t)\|_{L^1}\lesssim M_0+T_\vep,
	\end{equation}
	where $$Q(f)=\int_1^f q(\mu)\mu^{-2}\,d\mu.$$
\end{lemma}

\begin{proof}
	Taking the $L^2(\Omega)$ inner product on the second equation in $\eqref{FEL-final}$ with $v_i$ gives
	\begin{align*}
		&\frac{1}{2}\frac{d}{dt}\int_\Omega\tilde{\rho}_0|v|^2 dx+\int_\Omega a_{ik}\D_kqv_i dx-2\mu\vep\int_\Omega a_{kl}\D_l(S_\eta v)_{ik}v_idx\\
	&-\lambda\vep\int_\Omega a_{ij}\D_j(\Div_\eta v)v_i-\int_\Omega\D_j(\tilde{\rho}_0\D_j\eta_i) v_i dx=0.
	\end{align*}
	Using the integration by parts and   Piola's identity \eqref{Piola}, we have
	\begin{equation}
		\label{infd}
		\begin{split}
			&\int_\Omega a_{ik}\D_kqv_idx-2\mu\vep\int_\Omega a_{kl}\D_l(S_\eta v)_{ik}v_idx-\lambda\vep\int_\Omega a_{ij}\D_j(\Div_\eta v)v_i-\int_\Omega\D_j(\tilde{\rho}_0\D_j\eta_i) v_idx\\
	&=\underbrace{-\int_\Omega a_{ik}\D_kv_iqdx+2\mu\vep\int_\Omega(S_\eta v)_{ik}a_{kl}\D_lv_idx+\lambda\vep\int_\Omega J(\Div_\eta v)^2dx+\int_\Omega\tilde{\rho}_0\D_j\eta_i\D_jv_i dx}_{R_e}\\
    &\quad \underbrace{-\int_\Gamma q v_ia_{i3}+2\mu\vep\int_\Gamma (S_\eta v)_{ik}a_{k3}v_i+\lambda\vep\int_\Gamma a_{i3}v_i \Div_\eta v+\int_\Gamma\tilde\rho_0\D_3\eta_iv_i}_{R_b}.
		\end{split}
	\end{equation}
For $R_e$ in \eqref{infd}, by   \eqref{Jacobi's formula} and \eqref{bdNavier}, one has
	\begin{equation}
		\label{inner}
		\begin{split}
			&R_e=-\int_\Omega J_tqdx+\frac{1}{2}\frac{d}{dt}\int_\Omega\tilde{\rho}_0|\nabla \eta|^2dx+2\mu\vep\int_\Omega J|S_\eta(v)|^2\,dx+\lambda\vep\int_\Omega J|\Div_\eta v|^2\,dx
			\\
			&=\int_\Omega\tilde{\rho}_0f_tf^{-2}q(f)dx+\frac{1}{2}\frac{d}{dt}\int_\Omega\tilde{\rho}_0|\nabla \eta|^2dx+2\mu\vep\int_\Omega J|S_\eta(v)|^2\,dx+\lambda\vep\int_\Omega J|\Div_\eta v|^2\,dx
			\\
			&=\frac{d}{dt}\int_\Omega\tilde{\rho}_0Q(f)dx+\frac{1}{2}\frac{d}{dt}\int_\Omega\tilde{\rho}_0|\nabla \eta|^2dx+2\mu\vep\int_\Omega J|S_\eta(v)|^2\,dx+\lambda\vep\int_\Omega J|\Div_\eta v|^2\,dx,
		\end{split}
	\end{equation}
	where $Q(f)=\int_{1}^{f}q(\mu)\mu^{-2}d\mu$.

	For $R_b$, using the boundary condition \eqref{bdNavier} and \eqref{340}, we arrive at
	\begin{equation}
		\label{bdi}
		\begin{split}
		R_b=\sum_{\beta=1,2}\int_\Gamma \tilde\rho_0F_{i3}+2\mu\vep(S_\eta(v^\vep))_{\beta 3}a_{33} v_\beta^\vep=-\int_\Gamma \alpha \vep\abs{a_{33}}\abs{v}^2.
		\end{split}
	\end{equation}
	Therefore, we conclude that
	\begin{align*}
	\frac{d}{dt}\left(\frac{1}{2}\int_\Omega\tilde{\rho}_0(|v|^2+|\nabla\eta|^2+2Q(f))\right)&+2\mu\vep\int_\Omega J|S_\eta(v)|^2\,dx+\lambda\vep\int_\Omega J|\Div_\eta v|^2\,dx\\-\int_\Gamma \alpha\vep \abs{a_{33}}|v|^2=0.
	\end{align*}
	Then, integrating it in time yields
	\begin{align*}
    &\left(\frac{1}{2}\int_\Omega\tilde{\rho}_0(|v|^2+|\nabla\eta|^2+2Q(f))\right)(t)+2\mu\vep\int_0^t\int_\Omega J\left(|S_\eta(v)|^2-\dfrac{1}{3}|\Div_\eta v|^2\right)\,dx\\&+\left(\dfrac{2}{3}\mu+\lambda\right)\vep\int_0^t\int_\Omega J|\Div_\eta v|^2\,dx-\int_0^t\int_\Gamma \alpha\vep\abs{a_{33}}\abs{v}^2
    =\frac{1}{2}\int_\Omega\tilde{\rho}_0(|v_0|^2+|\nabla\eta_0|^2+2Q(f_0)).\nonumber
	\end{align*}
By the trace estimates \eqref{tre} and Young's inequality, we have
\begin{equation}
	\label{314}
	\abs{\int_\Gamma \alpha\vep\abs{a_{33}}\abs{v}^2}\ls \norm{v}_0^2+\vep^2\norm{\nabla v}_0^2.
\end{equation}
Moreover, since $\mu>0$, $2\mu+3\lambda>0$, and
    \begin{equation*}
            |S_\eta(f)|^2-\dfrac{1}{3}|\Div_\eta f|^2=|S_\eta(f)-\dfrac{1}{3}\Div_\eta f I|^2,
    \end{equation*}
 then   \eqref{Basic energy} follows from
    Korn's inequality,  together with \eqref{H1}, \eqref{A-priori-assum} and \eqref{314}.
\end{proof}

\subsection{Estimates of $q$}
Before deriving the higher-order estimates of solutions, we first prove the following key lemma for the estimate of pressure.
\begin{lemma}\label{est-J-q}
	For any $t\in[0,T_\vep]$, $m\geq 3$, 
	\begin{align}
		\sum_{j=0}^{m-1}\int_0^t\|\partial_t^j q^\vep\|_{m-j}^2\ls P\left(\sup_{t\in[0,T_\vep]}\mathfrak{E}^\vep(t)\right).\label{m-J-q}
	\end{align}
\end{lemma}
\begin{proof}
	For the pressure $q$, we have from \eqref{gamma-law} and \eqref{f-eqn},
	\begin{equation}\label{qfor}
		q=f^\gamma=\tilde{\rho}_0^\gamma J^{-\gamma}.
	\end{equation}
	Then, using the apriori assumptions \eqref{H1} and \eqref{A-priori-assum}, we have,  for $j\leq m-1$,
	\begin{align*}
		\int_0^t\|\partial_t^j q\|_{m-j}^2&\ls \|\tilde{\rho}_0^\gamma\|_{L^\infty}^2\int_0^t\|\partial_t^j (J^{-\gamma})\|_{m-j}^2\ls \int_0^t\norm{\partial_t^j\nabla\eta}_{m-j}^2\ls P\left(\sup_{t\in[0,T]}\mathfrak{E}^\vep(t)\right).
	\end{align*}

	
\end{proof}

\subsection{Tangential derivative estimates}
We are ready to derive the estimates of high order tangential derivatives.
\begin{lemma}\label{Tan-est}
	Let $m\geq 4$, for any $t\in[0,T_\vep]$, it follows that
	\begin{align}\label{tan-est-1}
	&\sum_{j=0}^m\|\partial_t^j\partial_\tau^{m-j}(v^\vep, \nabla\eta^\vep)\|_{0}^2(t)+\|\partial_t^j\partial_\tau^{m-j} q^\vep\|_{0}^2(t)+\int_0^t\vep\|\partial_t^j\partial_\tau^{m-j}\DD v^\vep\|_{0}^2 \nonumber\\
	&\ls M_0+\delta\sup_{t\in[0,T_\vep]}\mathfrak{E}^\vep(t)+\sqrt {T_\vep} P\left(\sup_{t\in[0,T_\vep]}\mathfrak{E}^\vep (t)\right).
	\end{align}
\end{lemma}
\begin{proof}
Applying $\bar{\D}^m:=\partial_t^j\partial_\tau^{m-j}$ on the second equation in \eqref{FEL-final} and taking the $L^2(\Omega)$ inner product with $\bar{\D}^m v_i$ yield
	\begin{align*}
	&\frac{1}{2}\frac{d}{dt}\int_\Omega\tilde{\rho}_0|\bar{\D}^m v|^2dx+\int_\Omega a_{ik}\D_k \bar{\D}^m q\bar{\D}^m v_idx-2\mu\vep\int_\Omega \bar{\D}^m(a_{kl}\D_l(S_\eta v)_{ik})\bar{\D}^m v_idx\\
	&-\lambda\vep\int_\Omega \bar{\D}^m(a_{ij}\D_j\Div_
	\eta v)\bar{\D}^m v_idx-\int_\Omega \D_j(\bar{\D}^m(\tilde{\rho}_0\D_j\eta_i))\bar{\D}^m v_idx\\
	&=-\underbrace{\int_\Omega[\bar{\D}^m,\tilde{\rho}_0]\D_tv_i\bar{\D}^m v_idx}_{R_\eta^1}-\underbrace{\int_\Omega[\bar{\D}^m,a_{ik}]\D_kq\bar{\D}^m v_idx}_{R_q^1}.
	\end{align*}
	Using integration by parts and   Piola's identity \eqref{Piola}, we have
	\begin{align}
			&\quad\int_\Omega a_{ik}\D_k \bar{\D}^m q\bar{\D}^m v_idx-2\mu\vep\int_\Omega \bar{\D}^m(a_{kl}\D_l(S_\eta v)_{ik})\bar{\D}^m v_idx\nonumber\\
	&\quad-\lambda\vep\int_\Omega \bar{\D}^m(a_{ij}\D_j\Div_\eta (v))\bar{\D}^m v_idx-\int_\Omega \D_j(\bar{\D}^m(\tilde{\rho}_0\D_j\eta_i))\bar{\D}^m v_idx\nonumber\\
	&=-\int_\Omega a_{ik}\bar{\D}^m q\bar{\D}^m \D_kv_idx+2\mu\vep\int_\Omega\bar{\D}^m((S_\eta v)_{ik}a_{kl})\D_l\bar{\D}^m v_idx\nonumber\\
    &\quad+\lambda\vep\int_\Omega \bar{\D}^m(a_{ij}\Div_\eta v)\D_j\bar{\D}^m v_idx+\int_\Omega \bar{\D}^m(\tilde{\rho}_0\D_j\eta_i)\D_j\bar{\D}^m v_idx\nonumber
	\\&\quad \underbrace{-\int_\Gamma \bar\D^m q a_{i3}\bar\D^m v_i+\lambda\vep \int_\Gamma \bar\D^m\left(a_{i3}\Div_\eta(v)\right)\bar\D^m v_i}_{R_{b1}}\nonumber\\&\quad\underbrace{+2\mu\vep\int_\Gamma \bar\D^m(a_{k3}S_\eta(v)_{ik})\bar\D^mv_i+\int_\Gamma\bar\D^m(\tilde\rho_0\partial_3\eta_i)\bar\D^mv_i}_{R_{b2}}.
		\label{infd2}
	\end{align}
	For $R_{b1}$, from the first boundary condition in \eqref{bdNavier} we see that it equals to zero.

	For the  four integrals defined on the   domain $\Omega$ in \eqref{infd2}, we have
	\begin{align}
			&\quad-\int_\Omega a_{ik}\bar{\D}^m q\bar{\D}^m \D_kv_idx+2\mu\vep\int_\Omega\bar{\D}^m((S_\eta v)_{ik}a_{kl})\D_l\bar{\D}^m v_idx\nonumber\\
    &\quad+\lambda\vep\int_\Omega \bar{\D}^m(a_{ij}\Div_\eta v)\D_j\bar{\D}^m v_idx+\int_\Omega \bar{\D}^m(\tilde{\rho}_0\D_j\eta_i)\D_j\bar{\D}^m v_idx
			\nonumber\\&=-\int_\Omega\bar{\D}^m q\bar{\D}^m(a_{ik}\D_kv_i)dx+\underbrace{\int_\Omega\bar{\D}^m q[\bar{\D}^m,a_{ik}]\D_kv_idx}_{R_\eta^2}+2\mu\vep\int_\Omega J|S_{\eta}(\bar{\D}^m v)|^2dx\nonumber\\
			&\quad+\lambda\vep\int_\Omega J\abs{\Div_\eta(\bar{\D}^m v)}^2dx+\underbrace{2\mu\vep\int_\Omega[\bar{\D}^m, a_{kl}] (S_\eta v)_{ik}\D_l\bar{\D}^m v_idx}_{R_\vep^1}\nonumber\\
			&\quad+\underbrace{\lambda\vep\int_\Omega[\bar{\D}^m,a_{ij}](\Div_\eta v)\D_j\bar{\D}^m v_idx}_{R_\vep^2}+\underbrace{2\mu\vep\int_\Omega[\bar{\D}^m, A_{ij}]\D_jv_kJS_\eta(\bar{\D}^m v)_{ik}}_{R_\vep^3}\nonumber\\
			&\quad+\underbrace{\lambda\vep\int_\Omega[\bar{\D}^m, A_{ij}]\D_j v_iJ\Div_\eta(\bar{\D}^m v)}_{R_\vep^4}+\int_\Omega\tilde{\rho}_0\D_j\bar{\D}^m\eta_i\D_j\bar{\D}^m v_idx+\underbrace{\int_\Omega[\bar{\D}^m,\tilde{\rho}_0]\D_j\eta_i\D_j\bar{\D}^m v_idx}_{R_\eta^3}\nonumber\\
			&=\frac{1}{2}\frac{d}{dt}\int_\Omega\tilde{\rho}_0|\bar{\D}^m\nabla\eta|^2dx-\int_\Omega\bar{\D}^m q\bar{\D}^m\D_tJdx+2\mu\vep\int_\Omega J|S_{\eta}(\bar{\D}^m v)|^2dx+\lambda\vep\int_\Omega J\abs{\Div_\eta(\bar{\D}^m v)}^2dx\nonumber\\&\quad+R_\eta^2+R_\eta^3+\sum_{i=1}^4R_\vep^i.
	\end{align}
Since \eqref{f-eqn} implies
\begin{equation}\label{q-eqn2}
	\D_tJ=\D_t(\tilde{\rho}_0f^{-1})=-\tilde{\rho}_0\frac{\D_tf}{f^2}=-\tilde{\rho}_0\frac{\D_tq}{q'(f)f^2}=-\frac{\tilde{\rho}_0\D_tq}{\gamma Af^{\gamma+1}}=-\frac{J^{\gamma+1}}{\gamma A\tilde{\rho}_0^\gamma}\D_tq,
\end{equation}
then
\begin{align*}\label{I-est}
	&-\int_\Omega\bar{\D}^m q\bar{\D}^m\D_tJdx=\frac{1}{\gamma }\int_\Omega\bar{\D}^m q\bar{\D}^m(\tilde{\rho}_0^{-\gamma}J^{\gamma+1}\D_t q)dx\nonumber\\
	&=\frac{1}{\gamma }\int_\Omega\tilde{\rho}_0^{-\gamma}J^{\gamma+1}\bar{\D}^m q\D_t\bar{\D}^m qdx+\frac{1}{\gamma }\int_\Omega\bar{\D}^m q[\bar{\D}^m,\tilde{\rho}_0^{-\gamma}J^{\gamma+1}]\D_tqdx\\
	&=\frac{1}{2\gamma }\frac{d}{dt}\int_\Omega\tilde{\rho}_0^{-\gamma}J^{\gamma+1}|\bar{\D}^m q|^2dx\underbrace{-\frac{\gamma+1}{2\gamma }\int_\Omega\tilde{\rho}_0^{-\gamma}J^\gamma J_t|\bar{\D}^m q|^2dx}_{R_q^2}+\underbrace{\frac{1}{\gamma }\int_\Omega\bar{\D}^m q[\bar{\D}^m,\tilde{\rho}_0^{-\gamma}J^{\gamma+1}]\D_tqdx}_{R_q^3}.\nonumber
\end{align*}
	Consequently, collecting the above equations together, we can get the following
	\begin{equation}\label{tang-est}
	\begin{aligned}
		&\frac{1}{2}\frac{d}{dt}\int_\Omega\tilde{\rho}_0(|\bar{\D}^m v|^2+|\bar{\D}^m\nabla\eta|^2)dx+\frac{1}{2\gamma }\frac{d}{dt}\int_\Omega\tilde{\rho}_0^{-\gamma}J^{\gamma+1}|\bar{\D}^m q|^2dx\\
		&\quad+2\mu\vep\int_\Omega J|S_{\eta}(\bar{\D}^m v)|^2dx+\lambda\vep\int_\Omega J\abs{\Div_\eta(\bar{\D}^m v)}^2dx\\
		&=-\left(\sum_{i=1}^{3}R_\eta^i+\sum_{i=1}^{3}R_q^i+\sum_{i=1}^{4}R_\vep^i\right)-R_{b2}.
	\end{aligned}	
	\end{equation}	
We turn to estimate each of the terms   $R_{b2}$, $R_\eta^i$, $R_q^i\ (i=1, 2, 3)$ and $R_\vep^j\ (j=1,2,3,4)$.

	Firstly, for $R_{b2}$, with the first boundary condition in \eqref{bdNavier} and \eqref{340}, we have
	\begin{align}
			R_{b2}=&\sum_{\beta=1,2}\int_\Gamma \bar\D^m\left(2\mu\vep(S_\eta(v))_{\beta k}a_{k3}+\tilde\rho_0\D_3\eta_\beta\right)\bar\D^mv_\beta\nonumber\\=&-\alpha\vep \int_\Gamma \bar\D^m (\abs{a_{33}}v)\cdot\bar\D^mv=\alpha\vep \int_\Gamma \abs{a_{33}}\abs{\bar\D^m v}^2+\alpha\vep \int_\Gamma \left[\bar\D^m, \abs{a_{33}}\right]v\cdot \bar\D^m v\label{wqq}.
	\end{align}
	Thus, with the expression of $a_{33}$ and the trace estimate \eqref{peadv}, we arrive at
	\begin{equation}
		\abs{\alpha\vep\int_\Gamma\abs{a_{33}}\abs{\bar\D^m v}^2}\ls P\left(\sqrt{\mathfrak F^\vep}\right)\left(\vep^2\norm{\bar\D^m v}_0^2+\vep^2\norm{\bar\D^m\nabla v}_0^2\right),
	\end{equation}
	and
	\begin{align}
		\abs{\int_\Gamma \left[\bar\D^m,\abs{a_{33}}\right]v\cdot\bar\D^m v} &\ls \abs{\bar\D^m\D_\tau\eta}_{-\hal}\abs{\partial_\tau\eta}_{\frac{3}{2}}\abs{\vep\bar\D^m v}_{\hal}\nonumber\\&\ls P\left(\sqrt{\mathfrak F^\vep}\right)\left(\varepsilon^2\norm{\bar\D^m v}_0^2+\varepsilon^2\norm{\nabla\bar\D^m v}_0^2+\norm{\bar\D^m\nabla\eta}_0^2\right).	\label{rd}
	\end{align}

	Next, for $R_{\eta}^i\ (i=1,2,3)$, we have
	\begin{align}
				\left|\int_{0}^{t}R_\eta^1\right|&\ls\int_0^t\|[\bar{\D}^m,\tilde{\rho}_0]\D_t v_i\|_0\|\bar{\D}^m v_i\|_0\nonumber\\
				&\ls\int_0^t\norm{\tilde\rho_0}_m\norm{\partial_t v}_{L^\infty}\norm{\bar\D^mv}_0+\norm{\bar \D\tilde\rho_0}_{L^\infty}\norm{\bar\D^mv}_0^2\nonumber\\
				&\lesssim TP\left(\sup_{t\in[0,T]}\mathfrak E^\epsilon(t)\right).\label{R1-est}
	\end{align}
	Similarly,  the following holds
	\begin{align}
		\left|\int_{0}^{t}R_\eta^2\right|&\lesssim \int_0^t\norm{\bar\D^m q}_0\norm{\left[\bar\D^m, a_{ik}\right]\partial_k v}_0\nonumber\\&\ls \int_0^t\norm{\bar\D^m q}_0\left(\norm{\bar\D^m\nabla\eta}_0\norm{\nabla v}_{L^\infty}+\norm{\bar\D^m\nabla\eta}_0\norm{a}_{L^\infty}+\sqrt{\mathfrak F^\vep}\norm{\bar\D\nabla v}_{L^\infty}\right)\nonumber\\&\lesssim \sqrt TP\left(\sup_{t\in[0,T]}\mathfrak E^\epsilon(t)\right).\label{R3-est}
	\end{align}
	For $R_{\eta}^3$, using integration by parts with respect to $t$ and Cauchy's inequality, we obtain
	\begin{align}\label{R4-est}
		\left|\int_{0}^{t}R_\eta^3\right|&\lesssim\left|\int_\Omega[\bar{\D}^m,\tilde{\rho}_0]\D_j\eta_i\D_j\bar{\D}^m \eta_i\Big|_0^t\right|+\left|\int_{0}^{t}\int_\Omega[\bar{\D}^m,\tilde{\rho}_0]\D_t\D_j\eta_i\D_j\bar{\D}^m \eta_i\,dxd\mathsf t\right|\nonumber\\
		&\ls M_0+\delta\|\bar{\D}^m\nabla\eta(t)\|_{0}^2+C_\delta \mathfrak F^\vep+\int_0^t\norm{\bar\D^{m}\nabla\eta}_0^2\norm{\tilde\rho_0}_m\nonumber\\
		&\ls M_0+\delta\|\bar{\D}^m\nabla\eta(t)\|_{0}^2+TP\left(\sup_{t\in[0,T]}\mathfrak E^\epsilon(t)\right).
\end{align}	

Now we  estimate $R_q^i\ (i=1,2,3)$.
For $R_q^1$, from  \eqref{est-J-q} we have
		\begin{align}\label{R2-est}
		\left|\int_{0}^{t}R_q^1\right|&\lesssim\int_0^t\norm{\left[\bar\D^m, a_{ik}\right]\partial_kq}_0\norm{\bar\D^mv}_0\nonumber\\&\ls\int_0^t\left(\norm{\bar\D^m\nabla\eta}_0\norm{\nabla q}_{L^\infty}+\norm{\bar\D^{m-1}\nabla q}_0\norm{\bar\D a}_{L^\infty}\right)\norm{\bar\D^m v}_0\nonumber\\&\lesssim \sqrt TP\left(\sup_{t\in[0,T]}\mathfrak E^\epsilon(t)\right).
	\end{align}
In view of \eqref{H1}, \eqref{A-priori-assum}, one can derive the following
	\begin{equation}\label{Rq1-est}
	\left|\int_{0}^{t}R_{q}^2\right|\lesssim TP\left(\sup_{t\in[0,T]}\mathfrak E^\epsilon(t)\right).
	\end{equation}
It follows from \eqref{m-J-q} that
	\begin{align}\label{Rq2-est}
	\left|\int_{0}^{t}R_{q}^3\right|&\ls \int_0^t\|\bar{\D}^m q\|_{0}\|[\bar{\D}^m,\tilde{\rho}_0^{-\gamma}J^{\gamma+1}]\D_tq\|_{0}\nonumber\\
	&\ls \int_0^t\|\bar{\D}^m q\|_{0}\left(\|\bar{\D}(\tilde{\rho}_0^{-\gamma}J^{\gamma+1})\|_{L^\infty}\|\bar\D^{m-1}\D_tq\|_{0}+\|\D_tq\|_{L^\infty}\norm{\bar{\D}^m(\tilde{\rho}_0^{-\gamma}J^{\gamma+1})}_{0}\right)\nonumber\\
	&\lesssim TP\left(\sup_{t\in[0,T]}\mathfrak E^\epsilon(t)\right).
	\end{align}
Finally, for $R_\vep^{1,2}$, one has
\begin{align}\label{Rvep1}
		\left|\int_{0}^{t}R_\vep^1+R_\vep^2\right|&\ls\int_0^t\norm{\sqrt\vep[\bar{\D}^m, a](A\nabla v)}_0\norm{\sqrt\vep\DD\bar{\D}^m v}_0\nonumber\\
		&\ls \delta\vep\int_0^t\norm{\DD\bar{\D}^m v}_{0}^2+TP\left(\sup_{t\in[0,T]}\mathfrak{E}^\vep(t)\right).
\end{align}
Similarly, we have
\begin{align}\label{Rvep2}
	\left|\int_{0}^{t}R_\vep^3+R_\vep^4\right|&\ls \int_0^t\|\sqrt\vep[\bar{\D}^m, A]\DD v\|_{0}\left(\|\sqrt\vep\sqrt{J}S_{\eta}(\bar{\D}^m v)\|_{0}+\|\sqrt\vep\sqrt{J}\Div_\eta(\bar{\D}^m v)\|_{0}\right)\nonumber\\
                                              &\ls\delta\vep\int_0^t\|\nabla\bar{\D}^m v\|_{0}^2+TP\left(\sup_{t\in[0,T]}\mathfrak{E}^\vep(t)\right).
\end{align}
Therefore, integrating \eqref{tang-est} over $[0,t]$,    then using \eqref{rd}, \eqref{R1-est}--\eqref{Rvep2} and Korn's inequality, we finish the proof of Lemma \ref{Tan-est}.
\end{proof}

\subsection{Normal derivative estimates}
In this subsection, we will derive the estimates of normal derivatives. To this end, we can first obtain  from the second equation in $\eqref{FEL-final}$, \eqref{qfor} and \eqref{Geo-iden-1}  that
\begin{align*}
	&-\tilde{\rho}_0J\Delta\eta_i-\gamma (\tilde{\rho}_0J^{-1})^{\gamma}a_{ik}a_{rs}\D_k\D_s\eta_r-\mu\vep a_{kl}a_{kj}\D_j\D_l v_i-(\mu+\lambda)\vep  a_{kl}a_{ij}\D_l\D_j v_k\\
	&=-\gamma (\tilde{\rho}_0J^{-1})^{\gamma-1}J^{-1}a_{ik}\D_k\tilde\rho_0+\D_j\tilde{\rho}_0J\D_j\eta_i-\tilde{\rho}_0J\D_tv_i+\mu\vep Ja_{kl}\D_l A_{kj}\D_jv_i\\
	&\quad+\mu\vep Ja_{kl}\D_lA_{kj}\D_jv_i+\lambda\vep Ja_{ij}\D_jA_{kl}\D_lv_k.
\end{align*}
As a consequence, we have
\begin{equation}\label{normal-eqn}
	-\mathcal{A}_{ij}\D_3^2\eta_j-\mu\vep a_{k3}a_{k3}\D_3^2v_i-(\mu+\lambda)\vep a_{i3}a_{j3}\D_3^2v_j= \mathcal{F}_i+\mathcal{G}_i,
\end{equation}
where
\begin{equation*}
	\mathcal{A}_{ij}=\tilde{\rho}_0J\delta_{ij}+\gamma (\tilde{\rho}_0J^{-1})^\gamma a_{i3}a_{j3}=\tilde\rho_0 J(\delta_{ij}+\gamma\tilde\rho_0^{\gamma-1}J^{-\gamma-1}a_{i3}a_{j3}),
\end{equation*}
\begin{align*}
	\mathcal{F}_i=&\sum\limits_{l\neq 3,\text{ or }j\neq 3}\left(\mu\vep a_{kl}a_{kj}\D_j\D_l v_i+(\mu+\lambda)\vep  a_{kl}a_{ij}\D_l\D_j v_k\right)\nonumber\\
	&+\mu\vep Ja_{kl}\D_l A_{kj}\D_jv_i+\mu\vep Ja_{kl}\D_lA_{kj}\D_jv_i+\lambda\vep Ja_{ij}\D_jA_{kl}\D_lv_k,
\end{align*}
and
\begin{align*}
        \mathcal{G}_i=&\sum_{l\neq3,\text{ or }j\neq 3}\gamma (\tilde{\rho}_0J^{-1})^\gamma a_{il}a_{rj}\D_{lj}^2\eta_r\nonumber\\
                  &-\tilde{\rho}_0J\D_t v_i-\tilde{\rho}_0J\D_1^2\eta_i-\tilde{\rho}_0J\D_2^2\eta_i-\gamma\tilde{\rho}^{\gamma-1}_0J^{-\gamma}a_{ik}\D_k \tilde{\rho}_0+J\DD\tilde{\rho}_0\cdot\DD\eta_i.		
\end{align*}
It is straightforward to check that $\mathcal{A}$ is symmetric, and
 $\mathcal{A}$ is positive definite. In fact, for any vector $\boldsymbol{x}\in \mathbb{R}^3$, we have
\begin{equation*}
	\boldsymbol{x}^{T}\mathcal{A}\boldsymbol{x}=\tilde\rho_0 J\abs{\boldsymbol{x}}^2+\gamma(\tilde\rho_0 J^{-1})^\gamma \abs{a_{\cdot 3}\cdot \boldsymbol{x}}^2.
\end{equation*}
Since $\mathcal{A}$ is positive definite, with a priori assumption \eqref{A-priori-assum}, one can estimate the normal derivatives of $\eta$  by using \eqref{normal-eqn}. Consequently, we have the following lemma.
\begin{lemma}\label{Nor-est}
	For any $t\in[0,T_\vep]$, $m\geq 4$, one has
	\begin{align}\label{nor-est}
		&\sum_{j=0}^{m-1}\norm{\sqrt{\vep}\partial_t^j\DD^2\eta^\vep(t)}_{m-1-j}^2+\sum_{j=0}^m\int_{0}^{t}\|\partial_t^j(\DD\eta^\vep, v^\vep,\vep\nabla v^\vep)\|_{m-j}^2\nonumber\\
		&\ls M_0+\delta\sup_{t\in[0,T_\vep]}\mathfrak{E}^\vep(t)+\sqrt T P\left(\sup_{t\in[0,T_\vep]}\mathfrak{E}^\vep(t)\right).
	\end{align}
\end{lemma}
\begin{pf}
	 Applying the operator $\bar{\D}^\beta$ with $|\beta|\leq m-1$ to \eqref{normal-eqn} yields that
	 	\begin{align}\label{2nd-ord-eqn}	
	 		&\underbrace{-\mathcal{A}_{ij}\bar{\D}^\beta\D_3^2\eta_j-\mu\vep a_{k3}a_{k3}\bar{\D}^\beta\D_3^2v_i-(\mu+\lambda)\vep a_{i3}a_{j3}\bar{\D}^\beta\D_3^2v_j}_{R_n}\nonumber\\
	 		&=[\bar{\D}^\beta,\mathcal{A}_{ij}]\D_3^2\eta_j+\mu\vep[\bar{\D}^\beta,a_{k3}a_{k3}]\D_3^2v_i+(\mu+\lambda)\vep[\bar{\D}^\beta,a_{i3}a_{j3}]\D_3^2v_j+\bar{\D}^\beta\mathcal{F}_i+\bar{\D}^\beta\mathcal{G}_{i}.
	 	\end{align}
	 Then, we square the left-hand side of \eqref{2nd-ord-eqn} and integrate it over $\Omega$ to obtain
	 \begin{align}\label{square}
		\int_{\Omega}|R_n|^2=&\int_\Omega \abs{\mathcal{A}\bar\D^\beta\D_3^2\eta}^2+\mu^2\vep^2\abs{\abs{a_{\cdot 3}}^2\bar\D^\beta\D_3^2 v}^2+(\mu+\lambda)^2\vep^2 \abs{\bar\D^\beta\D_3^2v\cdot a_{\cdot 3}}^2\abs{a_{\cdot 3}}^2\nonumber\\&+\dfrac{d}{dt}\int_\Omega\mu\vep \abs{a_{\cdot 3}^2}\mathcal{A}_{ij}\bar\D^\beta\D_3^2\eta_i\bar\D^\beta\D_3^2\eta_j-\int_\Omega \mu\vep\D_t(\abs{a_{\cdot 3}}^2\mathcal{A}_{ij})\bar\D^\beta\D^2_3\eta_i\bar\D^\beta\D_3^2\eta_j\nonumber\\&+\dfrac{d}{dt}\int_\Omega (\mu+\lambda)\vep \tilde\rho_0 J(1+\gamma\tilde\rho_0^{\gamma-1}J^{-\gamma-1}\abs{a_{\cdot 3}}^2)\abs{\bar\D^\beta\D_3^2\eta\cdot a_{\cdot 3}}^2\\&-\int_\Omega (\mu+\lambda)\vep \D_t(\mathcal{A}_{ij}a_{i3}a_{k3})\bar\D^\beta\D_3^2\eta_j\bar\D^\beta\D_3^2\eta_k+2\int_\Omega\mu(\mu+\lambda)\vep^2\abs{a_{\cdot 3}}^2\abs{\bar\D^\beta\D_3^2v\cdot a_{\cdot 3}}^2,\nonumber
	 \end{align}
%
where we used
\begin{equation*}
	\mathcal{A}_{ij}a_{i3}=\tilde\rho_0 J(1+\gamma\tilde\rho_0^{\gamma-1}J^{-\gamma-1}\abs{a_{\cdot 3}}^2)a_{j3}.
\end{equation*}

     Since $\mu>0, 2\mu+3\lambda>0$, then $\mu+\lambda>\dfrac{1}{3}\mu>0$. Using the fact that $\mathcal{A}$ is positive definite and {\it a priori} assumption \eqref{A-priori-assum}, we can   integrate \eqref{square} over time to obtain,
\begin{align}\label{nor-est-1}
	&\norm{\sqrt{\vep}\bar{\D}^\beta\D_3^2\eta}_0^2(t)+\int_0^t\left(\norm{\bar\D^\beta\partial_3^2\eta}_0^2+\norm{\vep \bar\D^\beta\D_3^2v}_0^2\right) \nonumber\\\ls& \norm{\sqrt\vep\bar\D^\beta\partial_3^2\eta}_0^2(0)+\int_{0}^{t}\norm{\D_t\bar\D\eta}_{L^\infty}\norm{\sqrt{\vep}\bar{\D}^\beta\D_3^2\eta}_0^2+\int_0^t\norm{[\bar{\D}^\beta,\mathcal{A}_{\cdot j}]\D_3^2\eta_j}_{0}^2\nonumber\\
	&+\int_0^t\vep^2\norm{[\bar{\D}^\beta, a_{\cdot 3}a_{\cdot 3}]\D_3^2 v}_{0}^2+\norm{\bar{\D}^\beta\mathcal{F}}_{0}^2+\norm{\bar{\D}^\beta\mathcal{G}}_{0}^2.
\end{align}
It is straightforward to check that
\begin{align}\label{nor-est-1.1}
	\int_{0}^{t}\norm{\D_t\bar\D\eta}_{L^\infty}\|\sqrt{\vep}\bar{\D}^\beta\D_3^2\eta\|_0^2\ls TP\left(\sup_{t\in[0,T]}\mathfrak{E}^\vep(t)\right),
\end{align}
\begin{align}\label{A-com-est}
		\int_0^t\norm{[\bar{\D}^\beta,\mathcal{A}_{\cdot j}]\D_3^2\eta_j}_{0}^2
		&\ls\int_0^t\norm{\bar{\D}\mathcal{A}_{\cdot j}}_{L^\infty}^2\norm{\bar\D^{\beta-1}\D_3^2\eta_j}_{0}^2+\norm{\D_3^2\eta_j}_{L^\infty}^2\norm{\bar{\D}^\beta\mathcal{A}_{\cdot j}}_{0}^2
		\nonumber\\&\ls TP\left(\sup_{t\in[0,T]}\mathfrak{E}^\vep(t)\right),
\end{align}
and
\begin{align}\label{nor-est-1.2}
	\int_0^t\vep^2\|[\bar{\D}^\beta, a_{\cdot 3}a_{\cdot 3}]\D_3^2v\|_{0}^2&\ls\int_0^t\norm{\bar{\D}(a_{\cdot 3}a_{\cdot 3})}_{L^\infty}^2\norm{\sqrt\vep\bar\D^{\beta-1}\D_3^2v}_{0}^2+\norm{\sqrt\vep\D_3^2v}_{L^\infty}^2\norm{\bar{\D}^\beta(a_{\cdot 3}a_{\cdot 3})}_{0}^2\nonumber\\
	&\ls TP\left(\sup_{t\in[0,T]}\mathfrak{E}^\vep(t)\right).
\end{align}
By using \eqref{A-priori-assum} and \eqref{tan-est-1}, we have
\begin{align}\label{nor-est-1.3}
	\int_0^t\norm{\bar{\D}^\beta{\mathcal{F}}}_{0}^2&\ls \int_0^t\vep\|aa\|_{L^\infty}^2\norm{\sqrt\vep\bar{\D}^{\beta+1}\DD v}_{0}^2+\norm{\vep[\bar{\D}^\beta,aa]\bar\D\DD v}_{0}^2\nonumber\\
	&\quad+\int_0^t\left(\|\DD^2\eta\|_{L^\infty}^2\|\vep\DD v\|_{m-1}^2+\|\DD v\|_{L^\infty}^2\|\vep\DD^2\eta\|_{m-1}^2\right)\nonumber\\
	&\ls M_0+\delta\sup_{t\in [0,T_\vep]}\mathfrak E^\vep(t)+\sqrt TP\left(\sup_{t\in[0,T_\vep]}\mathfrak{E}^\vep(t)\right).
\end{align}
where $\delta$ may be adjusted when necessary.

Similarly, we also have
\begin{align}\label{nor-est-1.4}
	\int_0^t\norm{\bar{\D}^\beta{\mathcal{G}}}_{0}^2&\ls\int_0^t\left(\norm{\bar\D^m\nabla\eta}_0^2\norm{\tilde\rho_0 J a a}_{L^\infty}^2+\norm{\bar\D^\beta\nabla\eta}_0^2\norm{\bar\D\nabla\eta}_{L^\infty}^2\right)(1+\norm{\tilde\rho_0}_m^2)\nonumber\\&\ls TP\left(\sup_{t\in[0,T]}\mathfrak{E}(t)\right).
\end{align}
%
Therefore, substituting  \eqref{nor-est-1.1}-\eqref{nor-est-1.4} into \eqref{nor-est-1} yields
\begin{equation}\label{2nd-nor}
	\norm{\sqrt{\vep}\bar{\D}^\beta\D_3^2\eta}_0^2(t)+\int_0^t\left(\norm{\bar\D^\beta\partial_3^2\eta}_0^2+\norm{\vep\bar\D^\beta\D_3^2v}_0^2\right) \ls M_0+\delta\sup_{t\in[0,T]}\mathfrak{E}(t)+\sqrt TP\left(\sup_{t\in[0,T]}\mathfrak{E}(t)\right).
\end{equation}
Next, for any $|\beta|\leq m-1-\ell,\ell\in \mathbb{N}$, applying $\bar{\D}^{\beta}\D_3^\ell$ to \eqref{normal-eqn}, using the same argument as in the proof of \eqref{2nd-nor}, we can successively obtain the following estimate for $\ell=1,2,\cdots, m-1$,
\begin{align*}
&\|\sqrt{\vep}\bar{\D}^\beta\D_3^{\ell+2}\eta\|_{0}^2+\int_0^t\left(\norm{\bar{\D}^\beta\D_3^{\ell+2}\eta}_{0}^2+\norm{\vep\bar{\D}^\beta\D_3^{\ell+2} v}_{0}^2\right)\ls M_0+\delta\sup_{t\in[0,T]}\mathfrak{E}(t)+\sqrt TP\left(\sup_{t\in[0,T]}\mathfrak{E}(t)\right).	
\end{align*}
Therefore, we complete the proof of Lemma \ref{Nor-est}.
\end{pf}

\subsection{Proof of Proposition \ref{uniform estimates} for Navier-slip boundary condition}
We now combine the estimates achieved above and verify the \textit{a priori} assumption \eqref{A-priori-assum}. In fact, we can deduce from Lemmas \ref{basic-energy-est}, \ref{Tan-est}, and \ref{Nor-est} that
\begin{equation*}
	\sup_{t\in[0,T_\vep]}\mathfrak{E}^\vep(t)\leq M_0+\delta\sup_{t\in[0,T_\vep]}\mathfrak{E}^\vep(t)+\sqrt{T_\vep}P\left(\sup_{t\in[0,T_\vep]}\mathfrak{E}^\vep(t)\right).
\end{equation*}
As a consequence, the following inequality holds for any $t\in[0,T_\vep]$,
\begin{equation*}
	|J^\vep(t)-J_0^\vep|\leq \left|\int_{0}^{t}J_t^\vep\right|\leq {T_\vep}^\frac{1}{2}\|J_t^\vep\|_{L^2_T(L^\infty)}\ls\sqrt{T_\vep}\sup_{t\in[0,T_\vep]}\mathfrak{E}^\vep(t).
\end{equation*}
Similarly, we also have
\begin{equation*}
	|\D_j\eta^\vep_i(t)-\D_j\eta^\vep_{0i}|\ls \sqrt{T_\vep}\sup_{t\in[0,T_\vep]}\mathfrak{E}^\vep(t).
\end{equation*}
Thus, by choosing $\delta$ sufficiently small, there exists a constant $T$ which is independent of $\vep$, such that \eqref{A-priori-assum} is satisfied and
\begin{equation*}
	\sup_{t\in[0,T]}\mathfrak{E}^\vep(t)\leq 2M_0.
\end{equation*}

\subsection{Proof of Proposition \ref{uniform estimates} for the no-slip boundary condition}
For no-slip boundary condition \eqref{bdDirichlet}, the only difference occurs in the estimates of the boundary integral in basic energy estimates and tangential derivative energy estimates. In fact,  we find that all of the boundary integral $R_b$ in \eqref{infd} and $R_{b1}, R_{b2}$ in \eqref{infd} become zero under the no-slip boundary condition. Thus, Lemma \ref{basic-energy-est} and \ref{Tan-est} are valid for no-slip boundary condition. Meanwhile the pressure and normal derivative estimates are just the same since the boundary conditions are not used in the proof of Lemma \ref{Nor-est}. Combining these three Lemmas, the Proposition \ref{uniform estimates} is also valid for the no-slip boundary condition.




\section{Proof of Theorems}

\subsection{Proof of Theorem \ref{Theorem 1}} 
From the uniform estimates of $(\eta^\vep,q^\vep, v^\vep)$ achieved in Proposition \ref{uniform estimates}, we find that there exists a $T_0>0$ which is independent of $\vep$, such that $(\eta^\vep,q^\vep, v^\vep)$ satisfy $\sup\limits_{t\in[0,T_0]}\mathfrak{E}^\vep(t)\leq C_1,$ which complete the proof of Theorem \ref{Theorem 1}. 

\subsection{Proof of Theorem \ref{vanishing viscosity1}} 
It follows from Proposition \ref{uniform estimates} that $\eta^\vep$ is uniformly bounded in $L^\infty(0,T_0;H^m)$, $\DD\eta^\vep$ is uniformly bounded in $L^2(0,T_0; H^m)$, and $\D_t\eta^\vep$ is uniformly bounded in $L^\infty(0,T_0;H^{m-1})\cap L^2(0,T_0;H^{m-1})$.  Then,  by using the Aubin-Lions compactness theorem (c.f. \cite{simon1986}), we obtain $\eta^\vep$ is compact in $C([0,T_0];H^{m-1})$. Precisely, there exist a subsequence $\vep_n\rightarrow 0^+$ and a function $\eta$, such that $\eta^{\vep_n}\rightarrow \eta$ in $C([0,T_0];H^{m-1})$ as $\vep_n\rightarrow 0^+$. Similarly, we can also have $v^{\vep_n}\rightarrow v$ in $C([0,T_0];H^{m-2})$ as $\vep_n\rightarrow 0^+$. Such convergence properties allow us to take the limit in \eqref{FEL-final} and prove that $(\eta,v)$ is a solution to the elastodynamic equations \eqref{FEL-final-2}. Thanks to the uniqueness of classical solutions to \eqref{FEL-final-2}, we have that the whole family $(\eta^\vep,v^\vep)$ converge to $(\eta,v)$.

\bigskip
	
\section*{Acknowledgments}
Xumin Gu was supported by National Natural Science Foundation of China (Grant No.
12031006) and the Shanghai Frontier Research Center of Modern Analysis.
 F.  Xie was supported by National Natural Science Foundation of China No. 12271359, 11831003, 12161141004, Shanghai Science and Technology Innovation Action Plan No. 20JC1413000 and Institute of Modern Analysis-A Frontier Research Center of Shanghai.
D. Wang was supported in part by National Science Foundation  grants  DMS-1907519 and DMS-2219384.

%
%
%
%
%
%
%
%

\end{document}